\documentclass[final,3p,mathptmx]{elsarticle}

\usepackage{amsmath,amsthm,amssymb,amsfonts}
\usepackage{upquote} 
\usepackage{xcolor}
\usepackage[normalem]{ulem}
\usepackage[hidelinks]{hyperref}

\newcommand{\stkout}[1]{\ifmmode\text{\sout{\ensuremath{#1}}}\else\sout{#1}\fi}

\newcommand{\RR}{{\mathbb R}}
\newcommand{\NN}{{\mathbb N}}
\newcommand{\ZZ}{{\mathbb Z}}

\newtheorem{theorem}{Theorem}[section]
\newtheorem{lemma}[theorem]{Lemma}
\newtheorem{proposition}[theorem]{Proposition}

\newtheorem{conjecture}[theorem]{Conjecture}

\newtheorem{definition}[theorem]{Definition}

\newtheorem{remark}[theorem]{Remark}
\newtheorem{example}[theorem]{Example}

\title{Multiple multivariate subdivision schemes: matrix and operator approaches}

\author[1]{Maria Charina}
\author[1]{Thomas Mejstrik}

\address[1]{Fakult\"at f\"ur Mathematik, Universit\"at Wien, Oskar-Morgenstern-Platz 1, 1090 Wien (Austria)}

\begin{document}

\begin{keyword} multiple subdivision schemes \sep convergence \sep joint spectral radius \sep  restricted spectral radius
\\[1em]
\noindent{\textit{Classification (MSC):} 65D17, 15A60}
\end{keyword}

\begin{abstract}
This paper extends the matrix based approach to the setting of multiple subdivision schemes
studied in~\cite{Sauer2012}.
Multiple subdivision schemes, in contrast to stationary and non-stationary schemes, allow for level dependent subdivision 
weights and for level dependent choice of the dilation matrices.
The latter property of multiple subdivision makes the standard definition of the transition matrices,
crucial ingredient of the matrix approach in the stationary and non-stationary settings, inapplicable. We show
how to avoid this obstacle and characterize the convergence of multiple subdivision schemes in terms of the
joint spectral radius of certain square matrices derived from subdivision weights. We illustrate our results with
several examples.
\end{abstract}

\maketitle

\section{Introduction}

The main contribution of this paper is the adaptation of the well known joint spectral radius 
approach (matrix based approach) to the setting of multiple subdivision. The connection between 
stationary (level independent) subdivision and joint spectral radius techniques was 
established in~\cite{DL}. In~\cite{CCGP16}, the authors introduced the matrix approach into the setting of non-stationary 
subdivision (with level dependent weights).  In both cases, the essential ingredient of the spectral radius technique are the so-called
transition matrices whose entries depend on the subdivision weights and whose structure is inherited from the
dilation matrix. The main challenge of adapting the matrix approach to the case of level dependent
dilation matrices is in combining their properties into an appropriate structure of the corresponding transition matrices. 

Subdivision schemes are recursive algorithms for generating meshes
in $\RR^d$, usually $d=2,3$. If the scheme is convergent, then the
sequence of meshes converges to a smooth limit curve or surface.
The vertices of these meshes are computed by means of weighted local
averages of the vertices of the mesh from the previous level of
the subdivision recursion. The topology of the regular mesh is
characterized by the integer dilation matrix $M \in \ZZ^{s \times
s}$ all of whose eigenvalues are larger than $1$ in the absolute
value. In the case of a \emph{multiple subdivision scheme}, the subdivision weights of local averages and the 
dilation matrices may vary depending on the level of the subdivision recursion. Multiple subdivision
schemes were introduced and studied in~\cite{Sauer2012}. 

The theory of subdivision schemes has had an impact on several applied
areas of mathematics and engineering and, in return, has been
greatly influenced by applications. First subdivision schemes with level independent subdivision weights 
and dilation matrices appeared in the `60s and are related to the
wavelet and frame theory whose applications are e.g. in signal and
image processing and in progressive geometry processing targeting
faster data transfer via internet. Recently, isogeometric
analysis and biological imaging promoted subdivision schemes with level-dependent weights 
by exploiting their ability to generate and reproduce exponential polynomials.
Multiple subdivision schemes are building blocks for processing of images with 
anisotropic directional features~\cite{Cotronei2015, KuSa} and for multigrid methods for solving anisotropic PDEs~\cite{CharinaDRT2}.

The most important properties of curves or surfaces generated by subdivision are their shape and smoothness.  
In the case of level independent or dependent subdivision weights, these properties are well understood. The 
subdivision shapes are characterized in terms of algebraic properties of subdivision symbols~\cite{CDM, CCRomani, JePlo, Jia, JJ, MoeSau}. 
The smoothness of 
subdivision is characterized either using the joint spectral radius~\cite{CHM, CCGP16, CP2017, CJR, DL, H1} 
or restricted spectral radius techniques~\cite{CDM, CCS04, DynLevin02}. 
Recent advances~\cite{GP1, GP2} in the exact computation of the joint spectral radius of compact sets of square matrices 
provide efficient methods for checking both H\"older and Sobolev regularity of subdivision. The study of the
properties of multiple subdivision is at its very beginning. The convergence analysis of multiple subdivision 
in terms of the restricted spectral radius is given in~\cite{Sauer2012}.  Our main result, Theorem~\ref{th:main}, relates
the convergence analysis of multiple subdivision and the joint spectral radius techniques and allows us
to use the efficient methods from~\cite{GP1}.

The paper is organized as follows. In section~\ref{sec:background} we recall the basic facts about
subdivision and multiple subdivision in particular. Section~\ref{sec_transition_matrices_supports} is devoted to
the construction of transition matrices with certain important invariance properties. These properties are crucial 
for our comparison, see section~\ref{sec:JSRvRSR},  of the restricted and joint spectral radius techniques in the context
of multiple subdivision. The applications of our theoretical results are given in section~\ref{sec:examples}.

\section{Background and notation} \label{sec:background}

\noindent The so-called matrix (or, the joint spectral radius) approach studies the spectral properties of
finite or compact sets of square matrices derived from the subdivision masks, see e.g.~\cite{CHM, CCGP16, CP2017, CJR, CH, DL, H1,HJ}.

\begin{definition}[\cite{Rota60}]\label{def_JSR}
 The \emph{joint spectral radius} of a finite set $\mathcal{A}$ of square matrices $A_r \in \mathcal{A}$ is defined by
\begin{equation}\label{equ_jsr}
  \rho(\mathcal{A})=\lim_{n\rightarrow\infty}\max_{A_r \in\mathcal{A}} \big\| \prod_{r=1}^n A_r \;\big\|^{1/n}.
\end{equation}
\end{definition}
\noindent The limit in the Definition~\ref{def_JSR} exists and is independent of the matrix norm~\cite[Proposition~1]{Rota60}.
The joint spectral radius quantifies the joint expanding properties of the matrices in~$\mathcal{A}$.

\subsection{Properties of the dilation matrices} \label{subsec:background_dilation}

\noindent  In the context of multiple subdivision the concept of the joint spectral radius is also used to describe
the joint expanding properties of several dilation matrices. This is a generalization of the standard requirement on the single dilation matrix to be expanding.

\begin{definition}\label{def_jointly_expanding}
A finite set of invertible matrices $\{M_j\in\ZZ^{s\times s}\;:\; j=1,\ldots,J\}$ is \emph{jointly expanding} if
	\begin{equation*}
	\rho(\{M_j^{-1}\;:\; j=1,\ldots,J\})<1.
	\end{equation*}
\end{definition}

\noindent Every dilation matrix $M_j$ has a possibly different digit set, which we define next.

\begin{definition}\label{def_digit_set} Let $j \in \{1, \ldots, J\}$.
A \emph{digit set $D_j\subset\ZZ^s$} corresponding to a dilation matrix $M_j$ is a complete set of representatives of the quotient group $\ZZ^s/M_j\ZZ^s=\{\alpha+M_j\ZZ^s:\alpha\in\ZZ^s\}$, i.e.\ $D_j\simeq \ZZ^s/M_j\ZZ^s$. The elements of a digit set are called \emph{digits}.
\end{definition}

\noindent We settle for the standard choice $D_j= \ZZ^s \cap M_j [0, 1)^s$, $j=1,\ldots,J$, implying that $0 \in D_j$. This choice of the digit sets does not necessarily lead to a tiling, but rather to a covering of $\RR^s$, see e.g.~\cite[section~2.2.2]{CHM}. We would like to emphasize that our results in section~\ref{sec:JSRvRSR} do not depend on the tiling property of the attractors.

\begin{definition}\label{def_attractor}
Let $\left\{ M_j\in\ZZ^{s\times s} \;:\; j=1,\ldots,J\right\}$ be jointly expanding with corresponding digit sets $D=\{D_j\subset\ZZ^s\;:\; j=1,\ldots,J\}$. We define the \emph{attractors} (subsets of $\RR^s$) associated to
$\boldsymbol{j}=(j_\ell)_{\ell \in \NN}$, $j_{\ell} \in \{1,\ldots,J\}$
\begin{equation}\label{equ_attractor}
K_{D\!,\,\boldsymbol{j}}=\operatorname{clos} \left(
M_{j_1}^{-1}D_{\!j_1}+M_{j_1}^{-1}M_{j_{2}}^{-1}D_{\!j_2}+\cdots \right)
= \operatorname{clos} \left(\sum_{r=1}^\infty  \Big(\prod_{\ell=1}^{r} M_{j_{\ell}}^{-1}\Big)\;D_{\!j_{r}} \right).
\end{equation}
\end{definition}

\noindent Note that the structure of the attractor $K_{D\!,\,\boldsymbol{j}}$ depends on the order of the indices in $\boldsymbol{j}$.

\noindent
The following properties of the attractors are reminiscent of the stationary and non-stationary settings.

\begin{lemma}\label{thm_attractor} 
Let 
$\left\{ M_j\in\ZZ^{s\times s} \;:\; j=1,\ldots,J\,\right\}$ be jointly expanding with corresponding digit sets 
$D=\{D_j\subset\ZZ^s\;:\; j=1,\ldots,J\,\}$ and
$\boldsymbol{j}=(j_\ell)_{\ell\in\NN}$, $j_\ell\in\{1,\ldots,J\}$.
Then $K_{D\!,\,\boldsymbol{j}}$ is compact.
\end{lemma}
\begin{proof} The boundedness of $K_{D\!,\,\boldsymbol{j}}$ follows, by~\cite[Proposition~1]{Rota60}, due to the existence of a matrix norm $\|\cdot\|$ such that
$\displaystyle C_1=\max_{j=1,\ldots,J}\|M_j^{-1}\|<1$, and the fact that the sets $D_j$, $j=1,\ldots,J$, are finite, i.e.\ bounded by $0<C_2< \infty$.
Indeed, for every $x \in K_{D\!,\,\boldsymbol{j}}$, we have
\begin{equation*}
\|x\|= \| M_{j_1}^{-1} d_{j_1} + M_{j_1}^{-1} M_{j_2}^{-1} d_{j_2}+\cdots\| 
 \leq  
 C_2 \sum_{\ell=1}^\infty C_1^{\ell} = \frac{C_2 C_1}{1-C_1}, 
 \quad 
 d_{j_\ell} \in D_{\!j_\ell}. 
 \qedhere
\end{equation*}
\end{proof}

\subsection{Multiple subdivision and its properties} \label{subsec:background_subdivision}

\noindent The definition of subdivision operators associated to finite sets of finitely supported masks and jointly expanding dilation matrices is done analogously to the stationary or the non-stationary case.

\begin{definition}\label{def_subdivision_operator} Let $j \in \{1, \ldots,J\}$, $J \in \NN$.
For a \emph{mask} $a_j\in\ell_0(\ZZ^s)$  and a dilation matrix $M_j \in\ZZ^{s\times s}$,  the \emph{subdivision operator} $S_j:\ell(\ZZ^s)\rightarrow \ell(\ZZ^s)$  defined by the pair $(a_j,M_j)$ is given by
\begin{equation}\label{equ_subdivision_opertor}
 S_j c(\alpha)=\sum_{\beta\in\ZZ^s} a_j(\alpha - M_j\beta)c(\beta),\quad\alpha\in\ZZ^s.
\end{equation}
\end{definition}

\noindent 
Without loss of generality, we assume that $0\in\operatorname{supp}(a_j)$, $j=1,\ldots,J$.

\noindent \vspace{0.2cm}
The concept of multiple subdivision schemes was introduced in~\cite{Sauer2012}.

\begin{definition} \label{def_subdivision_scheme} Given $\{a_j \in \ell_0(\ZZ^s)\;:\; j=1,\ldots,J\}$
and jointly expanding $\{M_j\in\ZZ^{s\times s}\;:\; j=1,\ldots,J\}$.
\begin{description}
\item[$(i)$] We define the \emph{finite set $\mathcal{S}$ of subdivision operators} $S_j$  by
\begin{equation} \label{def:setS}
\mathcal{S}=\left\{S_j=(a_j,M_j)\;:\;  a_j\in \ell_0(\ZZ^s),\; M_j\in \ZZ^{s\times s},\;j=1,\ldots,J\right\}.
\end{equation}
\item[$(ii)$] A sequence $(S_{j_\ell})_{\ell \in \NN}\in\mathcal{S}^\NN$, $j_\ell\in \{1,\ldots,J\}$,
is called a \emph{(multiple) subdivision scheme}.
\end{description}
\end{definition}

\begin{remark} The concept of multiple subdivision generalizes stationary and 
non-stationary settings. Indeed, 
the set $\mathcal{S}^\NN$ of all possible (multiple) subdivision schemes contains stationary subdivision 
schemes~-- the sequences $(S)_{\ell \in \NN} \in \mathcal{S}^\NN$ with
$S \in  \mathcal{S}$ defined by the pair $(a,M)$.  The set $\mathcal{S}^\NN$ also  includes 
certain non-stationary subdivision schemes~-- the sequences $(S_{j_\ell})_{\ell \in \NN} \in \mathcal{S}^\NN$ with the subdivision operators $S_{j_\ell}  \in  \mathcal{S}$ defined by the pairs $(a_{j_\ell},M)$. 
\end{remark}

\begin{definition}\label{def_ss_conv_to_Wkp_functions} Let $\mathcal{S}$
be a finite set of subdivision operators.
\begin{description}
\item[$(i)$] We say that a (multiple) subdivision scheme $(S_{j_\ell})_{\ell \in \NN}\in\mathcal{S}^\NN$ is \emph{convergent}
if for every sequence $c\in\ell_\infty(\ZZ^s)$ there exists a \emph{function} $g_{c,\boldsymbol{j}}\in C(\RR^s)$ (which is non-zero for at least one sequence $c$) such that
\begin{equation}\label{equ_ss_conv_to_C0_functions}
 \lim_{n\rightarrow\infty} \Big\| g_{c,\boldsymbol{j}}(M_{j_1}^{-1}\cdots M_{j_n}^{-1}\cdot) - S_{j_n}\cdots S_{j_1}c\Big\| _{\ell_\infty} = 0, \quad \quad
 \boldsymbol{j}=(j_\ell)_{\ell \in \NN}.
\end{equation}
\item[$(ii)$] We say that $\mathcal{S}^\NN$ is convergent, if every subdivision scheme in $\mathcal{S}^\NN$ is convergent.
\end{description}
\end{definition}

\begin{remark} 
For the limit function $g_{c,\boldsymbol{j}}$ in Definition~\ref{def_ss_conv_to_Wkp_functions} we write
\begin{equation}\label{equ_function_as_limit_of_subdiv_op}
 g_{c,\boldsymbol{j}} = \lim_{n\rightarrow \infty} S_{j_n}\cdots S_{j_2} S_{j_1} c, \quad
 \boldsymbol{j}=(j_\ell)_{\ell \in \NN}.
\end{equation}
\end{remark}

\noindent The necessary conditions, the sum rules of order one, for convergence of stationary subdivision schemes in $\mathcal{S}^\NN$ are well known, see e.g.~\cite{CDM,DynLevin02, JePlo, JJ}.

\begin{lemma}\label{thm_sum_rule_zero} Let $\mathcal{S}^\NN$ be convergent. Then every stationary subdivision
scheme defined by the pair $(a_j,M_j)$, $j \in \{1, \ldots,J\}$ is convergent and its mask $a_j$ satisfies
the sum rules of order one,
\begin{equation}\label{equ_sum_rule_zero}
   \sum_{\beta\in\ZZ^s} a_j(M_j\beta+\alpha)=1, \quad\alpha\in\ZZ^s.
\end{equation}
\end{lemma}

\noindent The result of Lemma~\ref{thm_sum_rule_zero} gives rise to the following assumption.

\smallskip
\noindent \textbf{Assumption S:} {\it We assume that the masks $a_j$, $j=1,\ldots,J$, in $\mathcal{S}$ satisfy sum rules of order one.}
\smallskip

\noindent Furthermore, if $\mathcal{S}^\NN$ is convergent, then every  (multiple) subdivision scheme in $\mathcal{S}^\NN$ possesses a sequence of basic limit functions. Similarly to the non-stationary setting, the concept of refinability is defined for the basic limit functions generated by the certain (multiple) subdivision schemes $(S_{j_\ell})_{\ell \ge r}\in\mathcal{S}^\NN$, $r \in \NN$, related by the shift in the ordering of the corresponding subdivision operators. To indicate this shift we introduce the following
sequence $\boldsymbol{j}^{[r]}$. 

\begin{definition}
Let $\boldsymbol{j}=(j_\ell)_{\ell\in\NN}$, $j_\ell\in\{1,\ldots,J\}$. For $r\in\NN$ we define a
shifted sequence 
$$\boldsymbol{j}^{[r]}=(j_r,j_{r+1},j_{r+2},\cdots)=(j_{\ell+r-1})_{\ell\in\NN}.$$
\end{definition}


\begin{definition}\label{def_basic_limit_function} 
For a (multiple) convergent subdivision scheme $(S_{j_\ell})_{\ell \in \NN}\in\mathcal{S}^\NN$, we define the sequence of \emph{basic limit functions}
\begin{equation}\label{equ_basic_limit_function}
 \phi_{\boldsymbol{j}^{[r]}} = 
 \lim_{n \rightarrow \infty}  S_{j_{r+n}}\cdots S_{j_{r+1}} S_{j_r} \delta
 ,\quad
 \delta(\alpha)=\left\{\begin{array}{rr} 1, & \alpha=0 \\ 0, & \text{otherwise} \end{array}\right.
 ,\quad
 r\in\NN.
\end{equation}
If the scheme $(S_{j_\ell})_{\ell \in \NN}$ is stationary, i.e. $S_{j_\ell}=S$ for all
$\ell \in \NN$, then $\phi_{\boldsymbol{j}^{[r]}}=\phi$ for all $r \in \NN$.
\end{definition}

\begin{remark} Note that $\phi_{\boldsymbol{j}^{[r]}}$, $r=2,3, \ldots$, by themselves are limits of certain subdivision schemes in $\mathcal{S}^\NN$.
\end{remark}

\noindent The proof of the mutual refinability of the functions $\phi_{\boldsymbol{j}^{[r]}}$, $r \in \NN$, is analogous to stationary or
non-stationary settings~\cite[Theorem~2.1]{CDM}.

\begin{lemma}\label{thm_blf_properties}
Let $(S_{j_\ell})_{\ell \in \NN}\in\mathcal{S}^\NN$ be a convergent subdivision scheme. Then its basic limit functions $\phi_{\boldsymbol{j}^{[r]}}$, $r \in \NN$,
are mutually refinable, i.e.\ they satisfy the system of refinement equations\
\begin{equation}\label{equ_refinement_equation_1}
\phi_{\boldsymbol{j}^{[r]}}(x)=
\sum_{\alpha\in\ZZ^s} a_{j_r}(\alpha)
 \; \phi_{\boldsymbol{j}^{[r+1]}} (M_{j_r} x -\alpha)
, \quad x\in\RR^s
,\quad r\in\NN.
\end{equation}
\end{lemma}

\noindent For a given $c\in\ell_\infty(\ZZ^s)$, the limit function
$g_{c,\boldsymbol{j}}$ in~\eqref{equ_function_as_limit_of_subdiv_op}
of the subdivision scheme $(S_{j_\ell})_{\ell \in \NN}\in\mathcal{S}^\NN$ can be written as a linear combination of the integer shifts of the
corresponding basic limit function $\phi_{\boldsymbol{j}^{[1]}}$. Thus, the convergence analysis of $\mathcal{S}^\NN$ is equivalent to the analysis of
uniform continuity of the corresponding basic limit functions. In section~\ref{sec_transition_matrices_supports}, we
show how to rewrite~\eqref{equ_refinement_equation_1} in an equivalent vector-valued form, where the summation
in~\eqref{equ_refinement_equation_1} is replaced by a matrix vector multiplication. To do that we need to gain
more insight about the structure of the supports of the basic limit functions. See e.g.~\cite{CohenDyn} for details
in the stationary and non-stationary settings.

\noindent The straightforward observation that the compact sets
\begin{equation} \label{def:KA}
K_{A,\,{\boldsymbol{j}^{[r]}}}=
\operatorname{clos} \left(\sum_{\ell=r}^\infty  \Big(\prod_{i=1}^{\ell} M_{j_{i}}^{-1}\Big)\;\operatorname{supp}(a_{j_{\ell}}) \right),  \quad A={\{\operatorname{supp}(a_j) \;:\; j=1, \ldots,J\}}, \quad r \in \NN,
\end{equation}
determine the re-parametrization  (see e.g.~\cite[(1.2)]{CCRomani} in the non-stationary case) for the
subdivision sequences that approximate the values of $(\phi_{\boldsymbol{j}^{[r]}})_{{r \in \NN}}$ implies the following result.

\begin{lemma}\label{thm_supp_refinement_equation}
Let $(S_{j_\ell})_{\ell \in \NN}\in\mathcal{S}^\NN$ be a convergent subdivision scheme
and  $A={\{\operatorname{supp}(a_j) \;:\; j=1, \ldots,J\}}$. Then the supports of the corresponding basic limit
functions $\phi_{\boldsymbol{j}^{[r]}}$, $r \in \NN$, satisfy
\begin{equation}\label{equ_support_1}
\operatorname{supp}(\phi_{\boldsymbol{j}^{[r]}}) \subseteq K_{A,\,{\boldsymbol{j}^{[r]}}}, \quad r \in \NN.
\end{equation}
Moreover, if the mask entries $a_j(\alpha)>0$, $\alpha \in \operatorname{supp}(a_j)$, $j=1, \ldots,J$, then $\operatorname{supp}(\phi_{\boldsymbol{j}^{[r]}}) = K_{A,\,{\boldsymbol{j}^{[r]}}}$, $r \in \NN$.
\end{lemma}

\noindent Example~\ref{ex_supp_of_blf}  shows that, for different subdivision schemes
in $\mathcal{S}^\NN$, the supports of the corresponding basic limit functions may have a
completely different structure. Similar observation has been already made in the context of non-stationary schemes in e.g.~\cite{CCGP16a}.

\begin{example}\label{ex_supp_of_blf} We consider
the set $\mathcal{S}=\{S_j=(a_j,M_j)\;:\; j=1,2\}$ from~\cite[section~4]{Cotronei2015} with the dilation matrices
$$
 M_1=\left(\begin{array}{rr}1 &1\\ 1& -2\end{array}\right), \quad
 M_2=\left(\begin{array}{rr}2 &-1\\ 1& -2\end{array}\right),
$$
and the masks
$$
a_j(0,-2)=a_j(0,2)=\frac{1}{3}, \quad  a_j(0,-1)=a_j(0,1)=\frac{2}{3}, \quad a_j(0,0)=1, \quad j=1,2.
$$
The matrices $M_1$ and $M_2$ are jointly expanding, due to $\|M_{j_1}^{-1}M_{j_2}^{-1}\|_2<1$
for all $j_1, j_2 \in \{1,2\}$. The supports of the basic limit functions
$$
 \phi_1=\lim_{n \rightarrow \infty} \prod_{\ell=1}^n (S_2^{2\ell} S_1) \delta, \quad
 \phi_2=\lim_{n \rightarrow \infty} \prod_{\ell=2}^n (S_2^{2\ell} S_1) S_2 S_2 \delta \quad
\text{and} \quad 
 \phi_3=\lim_{n \rightarrow \infty} \prod_{\ell=2}^n (S_2^{2\ell} S_1) S_2 \delta
$$
are given in Figure~\ref{fig_supp_of_blf}. The Matlab code to produce the figures is
\begin{verbatim}
S=getS('2_ex_CGRS');
blf({[1 2 2 1 2 2 2 2 1 2 2 2 2 2 2 1],[2]},S,'iterations',9)
blf({[  2 2 1 2 2 2 2 1 2 2 2 2 2 2 1],[2]},S,'iterations',9)
blf({[    2 1 2 2 2 2 1 2 2 2 2 2 2 1],[2]},S,'iterations',9)
axis equal; axis([-2.8 2.9 -3.2 3.2]); 
\end{verbatim}
\end{example}

\begin{figure}
	\centering
	\includegraphics[width=.7\textwidth,height=5cm,keepaspectratio]{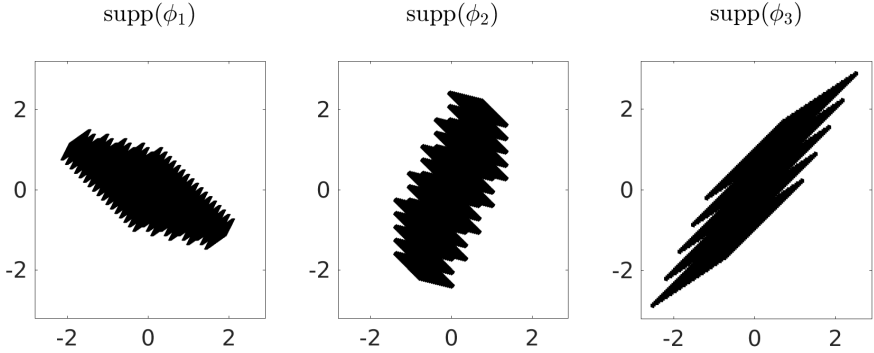} \\	
	\caption{Supports of the basic limit functions from Example~\ref{ex_supp_of_blf}.}
	\label{fig_supp_of_blf}
	%
	%
\end{figure}

\section{Transition matrices and matrix refinement} \label{sec_transition_matrices_supports}

\noindent In section~\ref{sec_transition_matrices_supports}, we construct the transition matrices in the setting of multiple multivariate subdivision. In particular, in Lemma \ref{alg_invariantomega}, we provide an algorithm for the construction of the minimal, invariant subspace  of the transition operators from Definition \ref{def_transition_operator}. The corresponding finite set $\Omega_C$ leads to transition matrices
of minimal size and, thus, is more suitable for computations in section \ref{sec:examples}. Furthermore, Lemma \ref{thm_invariantomega} together with Lemma \ref{lem:Support_OmegaZ_Attractor} guarantees the existence of a matrix vector form of the refinement equation~\eqref{equ_refinement_equation_1}. This explains the special role of the
finite set $\Omega_\ZZ$ constructed in Lemma \ref{thm_invariantomega}. The set $\Omega_\ZZ$ would be also suitable
for defining the transition matrices, but it cannot always be computed and would lead to transition matrices of a larger size.

\noindent Similarly to the stationary and
non-stationary settings, there are two important ingredients of our
construction: the transition operators and their common finite dimensional invariant subspaces.

\begin{definition}\label{def_transition_operator} Let $j \in \{1, \ldots, J\}$. For the subdivision mask $a_j \in \ell_0(\ZZ^s)$ and the dilation matrix $M_j \in \ZZ^{s \times s}$ with the digit set
$D_j\simeq \ZZ^s/M_j\ZZ^s$,  we define the \emph{transition operator} $\mathcal{T}_{d,j}: \ell(\ZZ^s)\rightarrow\ell(\ZZ^s)$ by
\begin{equation}\label{equ_transition_operator}
 (\mathcal{T}_{d,j} c)(\alpha)=\sum_{\beta\in\ZZ^s} a_j(M_j\alpha-\beta+d) \;c(\beta), \quad d\in D_j, \quad c \in\ell(\ZZ^s),\quad \alpha\in\ZZ^s.
\end{equation}
The set of all transition operators is denoted by
\begin{equation*}
 \mathcal{T}=\{\mathcal{T}_{d,j} \;:\; d \in D_j, \; j=1,\ldots,J\}.
\end{equation*}
\end{definition}

\noindent The result of Lemma~\ref{thm_invariantomega} ensures the existence of a common finite dimensional invariant subspace of the transition operators. 

\begin{definition}
Let $\Omega\subset\ZZ^s$. The set $\ell(\Omega)=\{c\in\ell(\ZZ^s) : \operatorname{supp} c \subseteq\Omega\}$ is the set of all sequences $c\in\ell(\ZZ^s)$ supported on $\Omega$.
\end{definition}

\begin{lemma}\label{thm_invariantomega}
There exists a finite set $\Omega_\ZZ \subset\ZZ^s$ such that  $\ell(\Omega_\ZZ)$ is invariant under all operators in $\mathcal{T}$.
\end{lemma}
\begin{proof}
By~\cite[Proposition~1]{Rota60} and due to the assumption that the dilation matrices are jointly expanding, there exists a matrix norm $\|\cdot\|$ such that the inverses of $M_j$ are contractive on $\RR^s$
w.r.t.\ this norm.
Let $\mathcal{X}$ be the set of all non-empty, compact subsets of $\RR^s$. By~\cite[2.10.21]{Federer}, the space $(\mathcal{X},d_H)$,
where $d_H$ is the Hausdorff metric w.r.t $\|\cdot\|$, is a complete metric space. The mappings
\begin{equation} \label{eq:def_contractiveMj}
\mathcal{M}_j: \mathcal{X} \rightarrow \mathcal{X}, \quad
\mathcal{M}_j(X)=M_j^{-1}(\operatorname{supp}a_j+X-D_j),\quad X\in \mathcal{X}, \quad
j=1,\ldots,J,
\end{equation}
are contractive. By the results in~\cite[section~3.1]{Hut81}, there exists a unique $\Omega_\RR\in \mathcal{X}$ such that
\begin{equation}\label{equ_one_invariant_vector_3}
 \Omega_\RR=\bigcup_{j=1,\ldots,J} \mathcal{M}_{j}(\Omega_\RR).
\end{equation}
Define
$$
 \Omega_\ZZ=\Omega_\RR\cap\ZZ^s.
$$
Let $d \in D_j$ for $j \in\{1,\ldots,J\}$. We show that $\mathcal{T}_{d,j}:\ell(\Omega_\ZZ)\rightarrow\ell(\Omega_\ZZ)$. Indeed, if
$v \in \ell(\Omega_\ZZ)$,
then by~\eqref{equ_transition_operator}
$ \displaystyle \operatorname{supp}(\mathcal{T}_{d,j} v)\subseteq
\bigcup_{j} M^{-1}_j(\operatorname{supp}a_j+\Omega_\ZZ-D_j)\subseteq
\bigg(\bigcup_{j} M^{-1}_j(\operatorname{supp}a_j+\Omega_\RR-D_j)\bigg) \cap \ZZ^s= \Omega_\ZZ
$.
\end{proof}

\noindent The result of Lemma~\ref{thm_invariantomega} allows us to associate each transition operator
$\mathcal{T}_{d,j}$ with a certain  square matrix.

\begin{definition} \label{def:transition_matrices} Let $\Omega \subset \ZZ^s$ be finite and such that $\ell(\Omega)$ is $\mathcal{T}$ invariant. For the operators in $\mathcal{T}$ we define the \emph{transition matrices}
\begin{equation*}
T_{d,j,\Omega}=\big(a_j(M_j\alpha-\beta+d)\big)_{\alpha,\beta\in\Omega}, \quad d \in D_j, \quad j=1, \ldots,J.
\end{equation*}
We note that $\alpha$ and $\beta$ are the respective row and column indices.
\end{definition}

\begin{remark} In the rest of the paper we use two other sets $\Omega \subset \ZZ^s$ such that $\ell(\Omega)$ is
invariant under all  operators in $\mathcal{T}$: the set $\Omega_C$ from Lemma~\ref{alg_invariantomega} for numerical computation in
section~\ref{sec:examples}; and the larger set $\Omega_V$ from Proposition~\ref{thm_VkequVbark_existence} for the
theoretical analysis in section~\ref{sec:JSRvRSR}.
\end{remark}

\noindent The sum rules of order one for the masks $a_j$, Assumption~\textbf{S}, become conditions on the spectral properties of the
transition matrices.

\begin{lemma}\label{thm_invariantone} Let $\mathcal{S}$ be a finite set of subdivision operators whose masks satisfy
Assumption~$\operatorname{\hbox{\bf S}}$ and $\mathcal{T}_{d,j} \in \mathcal{T}$.
\begin{description}
\item[$(i)$] If $\mathcal{T}_{d,j}:\ell(\Omega)
\rightarrow\ell(\Omega)$, $\Omega \subset \ZZ^s$, then the transition matrix $T_{d,j,\Omega}$ satisfies $(1,\ldots,1)\, T_{d,j,\Omega}=(1,\ldots,1)$.
\item[$(ii)$] If all entries of $T_{d,j,\Omega}$ are non-negative, then $(1,\ldots,1)\, T_{d,j,\Omega}=(1,\ldots,1)$ implies that
 $\mathcal{T}_{d,j}:\ell(\Omega)\rightarrow\ell(\Omega)$.
\end{description}
\end{lemma}
\begin{proof} $(i)$ Invariance of $\ell(\Omega)$ under $\mathcal{T}_{d,j}$, $d \in D_j$, $j \in \{1, \ldots,J\}$, implies,
by Definition~\ref{def_transition_operator}, that $a_j(M_j\alpha-\beta+d)=0$, whenever $\alpha \notin \Omega$ and $\beta\in \Omega$.
By~Assumption~\textbf{S}, we conclude that the entries in each column of the corresponding
transition matrix $\mathcal{T}_{d,j}$ sum up to one, since
\begin{equation*}
1=\sum_{\alpha\in\ZZ^s} a_j(M_j\alpha-\beta+d)=\sum_{\alpha\in\Omega} a_j(M_j\alpha-\beta+d), \quad \beta \in \Omega.
\end{equation*}
$(ii)$ Assume that $a_j(\alpha)\geq 0$ for all $\alpha\in\ZZ^s$, $j \in \{1,\ldots,J\}$. Due to $(1,\ldots,1)\, T_{d,j,\Omega}=(1,\ldots,1)$, $d \in D_j$, we get
\begin{equation*}
1=\sum_{\alpha\in\Omega}a_j(M_j\alpha-\beta+d),\quad \beta\in\Omega.
\end{equation*}
Assumption~\textbf{S}, i.e. the sum rules, implies that $a_j(M_j\alpha-\beta+d)=0$ for all $\alpha\not\in\Omega$, $\beta\in \Omega$.
Thus, $\mathcal{T}_{d,j}:\ell(\Omega)\rightarrow\ell(\Omega)$.
\end{proof}

\begin{remark} \label{rem_invariantomega} 
For the computation of the joint spectral radius in section~\ref{sec:examples}, the approximations (via the fixed point iteration~\cite[section~3.1~$(viii)$]{Hut81}) of $\Omega_\ZZ$ defined in Lemma~\ref{thm_invariantomega} are of no practical use. The following straightforward observation leads to an algorithm (see Lemma~\ref{alg_invariantomega}) for
explicit computation of $\Omega_C \subseteq \Omega_\ZZ$ with desired invariance properties as in Lemma~\ref{thm_invariantone}. Since $0\in\operatorname{supp}(a_j)$ and $0\in D_j$
for all $j=1,\ldots,J$, we conclude from~\cite[section~3.1~$(iii)$]{Hut81} that $0\in\Omega_\ZZ$.
Note that this set $\Omega_C$ is, by construction, the smallest $\mathcal{T}$ invariant set which contains $0$.
\end{remark}

\begin{lemma}\label{alg_invariantomega}  The following algorithm constructs a finite set $\Omega_C \subset \ZZ^s$ such that $\ell(\Omega_C)$ is $\mathcal{T}$ invariant.
\begin{flalign*}
&\Omega_0=\{0\},\ i=0 &\\
&\mathbf{repeat}\\
&\qquad   i=i+1\\
&\qquad   \Omega_i=\Omega_{i-1}\\
&\qquad   \mathbf{for}\ j=1,\ldots,J\ \mathbf{do}\\
&\qquad   \qquad   \Omega_{i,j}=(M_j^{-1}(\operatorname{supp}a_j + \Omega_i - D))\cap\ZZ^s\\
&\qquad   \qquad   \Omega_i=\Omega_i \cup \Omega_{i,j}\\
&\qquad   \mathbf{end}\\
&\mathbf{until}\  \Omega_{i}=\Omega_{i-1}\\
&\Omega_C=\Omega_i
\end{flalign*}

\end{lemma}
\begin{proof}
We first prove that the algorithm terminates after finitely many steps. More precisely, we show, by induction on $i$, that the sets
$(\Omega_i)_{i=0}^n$, $n \in \NN$, are increasing, nested subsets of the finite set $\Omega_\ZZ$ determined in
Lemma~\ref{thm_invariantomega}. Thus, $n$ is finite. Indeed, by Remark~\ref{rem_invariantomega},
$\Omega_0=\{0\}\subset\Omega_\ZZ$.
Assume that $\Omega_i\subseteq\Omega_\ZZ$, $i\le n$. By Lemma~\ref{thm_invariantomega}, $\Omega_\ZZ$ is invariant under all operators $M_j^{-1}(\operatorname{supp}(a_j)+\cdot-D)$,
thus, we get $\Omega_{i,j}\subseteq \Omega_\ZZ$ for all  $j=1,\ldots,J$.
Therefore, 
$
\Omega_{i}=\Omega_{i-1} \cup \bigcup_{j} \Omega_{i,j} \subseteq\Omega_\ZZ$.
Due to the stopping criterion, we get increasing, nested sets $\Omega_0\subset \Omega_1\subset  \Omega_2 \subset \cdots \Omega_n \subseteq\Omega_\ZZ$.
Moreover, for $\Omega_C=\Omega_n=\Omega_{n-1}$, due to~\eqref{equ_transition_operator}, we get
 $\mathcal{T}_d:\ell(\Omega_{n-1}) \rightarrow \ell(\Omega_{n})$, $d\in D_j$, $j\in\{1,\ldots,J\}$. Thus, the claim follows.
\end{proof}

\begin{remark}
The choice of $\Omega_0=\{0\}$ in Lemma~\ref{alg_invariantomega} is not crucial. Given any $\Omega_0\subset\ZZ^s$, finite, the algorithm in Lemma~\ref{alg_invariantomega} constructs a set $\Omega_C$ such that $\ell(\Omega_C)$ is $\mathcal{T}$ invariant and $\Omega_0\subseteq\Omega_C$. This follows directly by artificially enlarging the sets  $\operatorname{supp}a_j$ such that $\Omega_0\subseteq\Omega_\RR$, with $\Omega_{\RR}$ from Lemma~\ref{thm_invariantomega}.
\end{remark}

\noindent In some cases, the sets $\Omega_\ZZ$ defined in Lemma~\ref{thm_invariantomega} and $\Omega_C$ constructed in
Lemma~\ref{alg_invariantomega} coincide. 

\begin{example} $(i)$ Let $M=2$, $D=\{0,1\}$ and $\operatorname{supp}a=\{0,1,2\}$. Then
\begin{equation*}
\Omega_\RR= M^{-1}(\operatorname{supp}a-D)+M^{-2}(\operatorname{supp}a-D)+M^{-3}(\operatorname{supp}a-D)+\cdots=[-1,2]
\end{equation*}
Thus, $\Omega_\ZZ=\{-1,0,1,2\}$. The algorithm in Lemma~\ref{alg_invariantomega} generates $\Omega_C=\{0,1\}$.

$(ii)$ Let $M=-2$, $D=\{-1,~0\}$ and
$\operatorname{supp}a=\{0,1,2,3\}$. Then, by~\cite[Proposition~2.7]{CHM},
$\Omega_\RR=K_{A,\,\boldsymbol{j}}-K_{D\!,\,\boldsymbol{j}}=
[-\frac{5}{3},\frac{1}{3}]-[-\frac{1}{3},\frac{2}{3}]=
[-\frac{7}{3},\frac{2}{3}]$. Thus, $\Omega_\ZZ=\{-2,-1,0\}$.
The algorithm in Lemma~\ref{alg_invariantomega} produces the same set.
\end{example}

\noindent Similarly to the stationary and non-stationary settings, the supports of the basic limit functions
$(\phi_{\boldsymbol{j}^{[r]}})_{r \in \NN}$ can be covered by the integer shifts of the corresponding attractors $K_{D\!,\,\boldsymbol{j}^{[r]}}$
in Definition~\ref{def_attractor}. This leads to a standard matrix form of the refinement equations used for
analysing the existence and regularity of refinable functions in the stationary and non-stationary settings. 
The results of section~\ref{sec:JSRvRSR} however do not rely on such representations and the remaining part of
this section is merely for a curious reader.

\begin{lemma}\label{lem:Support_OmegaZ_Attractor}
Let $\Omega_\ZZ \subset \ZZ^s$ be as in Lemma~\ref{thm_invariantomega}. Assume that the subdivision scheme
$(S_{j_\ell})_{\ell \in \NN}\in\mathcal{S}^\NN$ is convergent. 
Then $\operatorname{supp}(\phi_{\boldsymbol{j}^{[r]}}) \subseteq \Omega_\ZZ+ K_{D\!,\,\boldsymbol{j}^{[r]}}$
for all $r\in\NN$.
\end{lemma}
\begin{proof}
Without loss of generality we assume $r=1$, i.e.\ $\boldsymbol{j}=\boldsymbol{j}^{[1]}$. Recall, from Lemma~\ref{thm_invariantomega}, that $\Omega_\RR \subset \RR^s$ is compact and
is the unique solution of the fixed point equation in~\eqref{equ_one_invariant_vector_3}.
We show first that $K_{A,\,\boldsymbol{j}}-K_{D\!,\,\boldsymbol{j}} \subseteq \Omega_\RR$ with
$K_{A,\,\boldsymbol{j}}$ in~\eqref{def:KA}, $K_{D\!,\,\boldsymbol{j}}$ in~\eqref{equ_attractor} and $\boldsymbol{j}=(j_\ell)_{\ell \in \NN}$. By the results
in~\cite[section~3.1]{Hut81}, the set $\Omega_\RR$ is also the closure (in the Hausdorff metric) of the fixed points of the compositions of the contractive mappings
${\cal M}_{r_i}$ defined in~\eqref{eq:def_contractiveMj}. More precisely,
\begin{equation*}
 \Omega_\RR= 
\operatorname{clos} \left\{ 
\Omega_{\ell_1,\ldots, \ell_k}  \in\mathcal{X} \;:\; 
\Omega_{\ell_1,\ldots, \ell_k} = 
{\cal M}_{\ell_1} \circ \ldots \circ {\cal M}_{\ell_k}(\Omega_{\ell_1,\ldots, \ell_k}), 
 \ \ell_i \in \{1, \ldots,J\}, \  
 i=1,\ldots,k, \ k \in \NN 
\, \right\}
\end{equation*}
and $\displaystyle \lim_{k \rightarrow \infty} \Omega_{\ell_1,\ldots, \ell_k}$ exist and belong to $\Omega_\RR$.
Thus,  for the specific ordering in $\boldsymbol{j}$, by~\eqref{eq:def_contractiveMj} and due to
\begin{align*}
  \Omega_\RR \supseteq {\cal M}_{j_1} \circ \ldots \circ {\cal M}_{j_k}(\Omega_{j_1,\ldots, j_k})
	 = M_{j_1}^{-1} \operatorname{supp}a_{j_1} + M_{j_1}^{-1} M_{j_2}^{-1} \operatorname{supp}a_{j_2}+ \ldots + M_{j_1}^{-1} \ldots M_{j_k}^{-1}
	 \operatorname{supp}a_{j_k} +&\\
   + M_{j_1}^{-1} \ldots M_{j_k}^{-1}\Omega_{j_1,\ldots, j_k}  
   - M_{j_1}^{-1} D_{\!j_1} - M_{j_1}^{-1} M_{j_2}^{-1} D_{\!j_2}+ \ldots - M_{j_1}^{-1} \ldots M_{j_k}^{-1} D_{\!j_k},&
\end{align*}
we get $K_{A,\,\boldsymbol{j}}-K_{D\!,\,\boldsymbol{j}} \subseteq \Omega_\RR$. Now we are ready to prove the claim. By~\cite[Lemma~1]{Groe2006},
$\RR^s=K_{D\!,\,\boldsymbol{j}}+\ZZ^s$. Thus, for $x\in\operatorname{supp}(\phi_{\boldsymbol{j}})=K_{A,\,\boldsymbol{j}}$, there exists $\alpha\in\ZZ^s$ such that
$x\in K_{D\!,\,\boldsymbol{j}}+\alpha$. Therefore, $\alpha\in x-K_{D\!,\,\boldsymbol{j}}\subseteq K_{A,\,\boldsymbol{j}}-K_{D\!,\,\boldsymbol{j}}\subseteq\Omega_\RR$,
i.e.\ $\alpha \in \Omega_\ZZ=\Omega_\RR \cap \ZZ^s$. This implies that $x \in K_{D\!,\,\boldsymbol{j}}+\Omega_\ZZ$.
\end{proof}

\noindent Lemma~\ref{lem:Support_OmegaZ_Attractor} generalizes the result~\cite[Proposition~2.7]{CHM}. We conjecture that the result of Lemma~\ref{lem:Support_OmegaZ_Attractor} is true
for an arbitrary finite $\Omega \subset \ZZ^s$, such that $\ell(\Omega)$ is $\mathcal{T}$ invariant, e.g. the
set $\Omega_C$ from Lemma~\ref{alg_invariantomega}.

\begin{conjecture} \label{lemma:supp_phi} Let $\mathcal{S}^{\NN}$ be convergent and $\Omega\subset\ZZ^s$ be finite and
such that $\ell(\Omega)$ is $\mathcal{T}$ invariant. 
Then $\operatorname{supp}(\phi_{\boldsymbol{j}^{[r]}}) \subseteq \Omega+K_{D\!,\,\boldsymbol{j}^{[r]}}$ for all  
$r\in\NN$.
\end{conjecture}

\section{Comparison of matrix and operator approaches: convergence of multiple subdivision schemes} \label{sec:JSRvRSR}

\noindent The goal of this section is to unify the matrix (joint spectral radius) and operator (restricted spectral radius)
approach in the setting of multiple subdivision schemes, see Theorem~\ref{th:main}. It generalizes
similar results in~\cite{Charina, CCS05} that were proven in the stationary setting for the case of the dilation
matrix $M=2I$.

\noindent One of the standard tools for checking the regularity of subdivision schemes is the so-called restricted
spectral radius (see e.g.~\cite{CDM, Charina, CCS04, Sauer2012}) that measures the spectral properties of the
difference subdivision operators restricted to a certain subspace of $\ell(\ZZ^s)$. 

The concept of the restricted
spectral radius relies on the difference operators and difference subdivision schemes operating on the sequences in $\ell(\ZZ^s)$. By $\ell(\ZZ^s,\RR^s)$ we denote the space of vector-valued (with values in $\RR^s$) sequences indexed by $\ZZ^s$.

\begin{definition} Let $e_\ell$, $1\leq \ell \leq s$, be the standard unit vectors of $\RR^s$. We define
\begin{description}
 \item[$(i)$] \emph{the $\ell$-th backward difference operator} $\nabla_{\!\ell}:\ell(\ZZ^s)\rightarrow\ell(\ZZ^s)$ \ by \
 $\nabla_{\!\ell}\, c=c-c(\cdot-e_\ell)$, \ $c \in \ell(\ZZ^s)$.
\item[$(ii)$] \emph{the backward difference operator} $\nabla:\ell(\ZZ^s)\rightarrow\ell(\ZZ^s,\RR^{s})$ \ by \
$ \displaystyle \nabla= \left(\begin{array}{cccc}
\nabla_{\!1}&
\nabla_{\!2}&
\dots&
\nabla_{\!s} \end{array} \right)^{\!T}$.
\end{description}
\end{definition}

\noindent The existence of difference subdivision operators (see e.g. for details~\cite{CDM, MoeSau}) is ensured by Assumption~\textbf{S}.

\begin{definition} Let $\mathcal{S}$ be a finite set of subdivision operators whose masks satisfy Assumption~$\operatorname{\hbox{\bf S}}$. For $S \in \mathcal{S}$,
\emph{a difference subdivision operator} $S':\ell(\ZZ^s) \rightarrow \ell(\ZZ^s, \RR^s)$ is defined by
\begin{equation}\label{def:Sprime}
 \nabla S=S' \nabla \, .
\end{equation}
By $\mathcal{S}'$ we denote a set of the difference operators $S'$ associated to the set $\mathcal{S}$ of subdivision operators.
 \end{definition}

\noindent In the setting of the multiple subdivision, we use the following definition of the restricted
spectral radius given in~\cite[section~3, ``\emph{normalized joint spectral radius}'']{Sauer2012}.

\begin{definition} \label{def:RSR} Let $\mathcal{S}$ be a finite set of subdivision operators whose masks satisfy
Assumption~\textbf{S}. The \emph{restricted norm} of $S' \in \mathcal{S}'$ is defined by
\begin{equation}\label{equ_def_RSR_norm}
	\|S'|_{\nabla}\|_\infty= \max_{\substack{ \|\nabla c \|_\infty=1}} \| S' \nabla c\|_{\infty}.
\end{equation}
The \emph{restricted spectral radius} of $\mathcal{S}'$ is defined by
\begin{equation}\label{equ_RSR}
\begin{aligned}
	\rho(\mathcal{S}'|_\nabla)=& \limsup_{n\rightarrow\infty} \sup_{ S'_{j_\ell} \in \mathcal{S}'}
		\left\|S'_{j_n}\cdots S'_{j_1}|_{\nabla}\right\|_\infty^{1/n}.
\end{aligned}
\end{equation}
\end{definition}

\noindent The main result of this section, Theorem~\ref{th:main}, leads to a characterization of convergence of $\mathcal{S}^\NN$ in terms of the joint spectral radius of the transition matrices (Definition~\ref{def:transition_matrices}) restricted to a common invariant subspace.
This characterization follows from Theorem~\ref{th:main} and the following result.

\begin{theorem}{\cite[Theorem~2]{Sauer2012}}\label{thm_RSR_convergence} $\mathcal{S}^\NN$ is convergent if and only if there exists $\mathcal{S}'$ such that
$\rho(\mathcal{S}'|_\nabla)<1$.
\end{theorem}

\noindent Theorem~\ref{th:main} allows us to use the invariant polytope algorithm from~\cite{GP1} for the computation of
the joint spectral radius, when checking the convergence of multiple subdivision schemes, see section~\ref{sec:examples}.
The proof of Theorem~\ref{th:main} is similar to the one of~\cite[Proposition~4.6]{Charina}, see also~\cite{CCS05}. The crucial differences between stationary
and multiple cases are pointed out in  Propositions~\ref{thm_different_RN} and~\ref{thm_VkequVbark_existence}.
Proposition~\ref{thm_different_RN} is a generalization of~\cite[Proposition~4.1]{Charina}.

\begin{proposition}\label{thm_different_RN} Let $(S'_{j_\ell})_{\ell \in \NN}\in\mathcal{S'}^\NN$ be a difference subdivision scheme. Then
	\begin{equation}\label{equ_rth_restricted_normalized_norm}
	\|S'_{j_n}\cdots S'_{j_1}|_{\nabla}\|_\infty=\max_{\substack{\nabla c\in\ell_\infty(([-1,1]^s-K)\cap \ZZ^s)\\
			\|\nabla c\|_\infty = 1}}\, \max_{\alpha\in M_{j_n}\cdots M_{j_1}[0,1)^s\cap\ZZ^s}\,
		\left\|S'_{j_n}\cdots S'_{j_1} \nabla c(\alpha)\right\|_{\infty}, \quad n \in \NN,
	\end{equation}	
	where
	\begin{equation} \label{def:K}
	K=\bigcup_{j=1,\ldots,J} M_j^{-1}(\operatorname{supp}a_j + K).
	\end{equation}
	\end{proposition}
	
\begin{proof} By definition of $S'_{j_\ell}$ 
we get
$$
		\|S'_{j_n}\cdots S'_{j_1}|_{\nabla}\|_\infty
		=\max_{\substack{ \|\nabla c\|_\infty = 1}}\,
		\sup_{\alpha\in\ZZ^s}\, \left\| \sum_{\beta\in\ZZ^s}  (S'_{j_n}\cdots S'_{j_1}\delta I)(\alpha - M_{j_n}\cdots M_{j_1} \beta)\nabla c(\beta)
		\right\|_\infty.
$$
Due to the periodicity of the subdivision (i.e.\ $|M_{j_n}\cdots M_{j_1}[0,1)^s\cap\ZZ^s|$ different subdivision rules at
the $n$-th
level of subdivision recursion), it suffices to take $\alpha\in M_{j_n}\cdots M_{j_1}[0,1)^s\cap\ZZ^s$. Finally, by~\cite[Remark~3.7]{Charina}
$$
 \operatorname{supp}(S'_{j_n}\cdots S'_{j_1}\delta I) \subseteq 
 M_{j_n}\cdots M_{j_2} \operatorname{supp}a_{j_1}+ \cdots +
 M_{j_n} \operatorname{supp}a_{j_{n-1}}+ \operatorname{supp}a_{j_n}.
$$
For $\alpha-M_{j_n}\cdots M_{j_1}\beta \in \operatorname{supp}(S'_{j_n}\cdots S'_{j_1}\delta I)$, by a similar argument as in the proof of Lemma~\ref{lem:Support_OmegaZ_Attractor} and due to $0\in\operatorname{supp}a_j$,  we obtain
\begin{align*}
	\beta &\in [0,1)^s\cap M_{j_1}^{-1}\cdots M_{j_n}^{-1}\ZZ^s - M_{j_1}^{-1} \operatorname{supp}a_{j_1} - \cdots -M_{j_1}^{-1}\cdots
	M_{j_n}^{-1} \operatorname{supp}a_{j_n}\\
&	\subseteq ([-1,1]^s-K) \cap \ZZ^s,
\end{align*}
where $K \subset \RR^s$ is the unique compact set satisfying the fixed point equation~\eqref{def:K}.
\end{proof}

\noindent Sufficient conditions for continuity of refinable functions or characterizations of continuity of
basic limit functions of subdivision schemes are usually formulated in terms of the spectral properties of
restrictions of transition matrices to $V_\Omega$ in~\eqref{equ_def_Vk} or restrictions of difference subdivision operators to $\tilde{V}_\Omega$ in~\eqref{equ_def_Vkbar}, respectively.

\begin{definition}\label{def_V0} Let $\Omega\subset\ZZ^s$ be finite with $n=|\Omega|$. We define the linear spaces
\begin{equation}\label{equ_def_Vk}
V_\Omega=\big\{ v\in \RR^n\;:\; \sum_{\beta\in\Omega} v(\beta)=0 \big\},
\end{equation}
\begin{equation}\label{equ_def_Vkbar}
\tilde{V}_\Omega= \operatorname{span}\left\{ v \in \ell_0(\ZZ^s) \;:\; v= \nabla \delta (\cdot-\beta), \ \beta\in \ZZ^s, \ \operatorname{supp}v \subseteq \Omega
\right\}.
\end{equation}
\end{definition}

\noindent In the rest of the paper, we view $\tilde{V}_\Omega$ as a subspace of $\RR^n$ (or, equivalently, $V_\Omega$ as a subspace of $\ell_0(\ZZ^s)$) and make use of the following properties of $V_\Omega$ and $\tilde{V}_\Omega$.

\begin{lemma}\label{thm_VkequVbark_dim}
Let $\Omega\subset\ZZ^s$ be  finite. Then $\tilde{V}_\Omega\subseteq V_\Omega$ and, if $\operatorname{dim}V_\Omega=
\operatorname{dim}\tilde{V}_\Omega$, then $V_\Omega = \tilde{V}_\Omega$.
\end{lemma}

\noindent We are now ready to formulate the main result of this section,
Theorem~\ref{th:main}.  The proof of Theorem~\ref{th:main} is given in subsection~\ref{subsec:proof_main}.

\begin{theorem} \label{th:main}
 Let $\mathcal{S}$ be a finite set of subdivision operators whose masks satisfy Assumption~$\operatorname{\hbox{\bf S}}$.
 Assume that there exists a finite set $\Omega \subset \ZZ^s$ such that
 \begin{itemize}
  \item[$(i)$] $\ell(\Omega)$ is invariant under the transition operators in $\mathcal{T}$ and
  \item[$(ii)$] $V_\Omega = \tilde{V}_\Omega$.
 \end{itemize}
 Then $\rho(\mathcal{S}'|_\nabla)=\rho(\{T_{d,j,\Omega}|_{V_\Omega} \;:\; d \in D_j, \ j=1,\ldots,J\})$.
\end{theorem}

\noindent  Example~\ref{ex_mult1} shows that assumption $(ii)$ of Theorem~\ref{th:main} is indeed crucial.
The natural candidate for such a set $\Omega$ would be the set $\Omega_C$ from Lemma~\ref{alg_invariantomega}.
The set $\Omega_C$, by Lemma~\ref{alg_invariantomega}, satisfies assumption $(i)$ of Theorem~\ref{th:main}
and our numerical experiments show that in most cases $\Omega_C$ also satisfies the assumption $(ii)$.
However, Example~\ref{ex_V0_neq_V0bar} illustrates that the case $\tilde{V}_{\Omega_C} \subset V_{\Omega_C}$
occurs sometimes even in the stationary setting.
In such cases, we choose $\Omega=\Omega_V$ from Proposition~\ref{thm_VkequVbark_existence}.

\begin{example}\label{ex_V0_neq_V0bar}
%
%
%

Consider the dilation matrix  $M=\left(\begin{array}{rr}-3& -4\\ 4& 4\end{array}\right)$ with the digit set
$D=\{ (-k,k) : k=0,1,2,3 \}$ and choose any mask $a$ with
$$
 \operatorname{supp}(a)=\left\{ \left(\begin{array}{c}1\\0\end{array}\right),
\left(\begin{array}{c}2\\1\end{array}\right),
\left(\begin{array}{c}3\\1\end{array}\right),
\left(\begin{array}{c}1\\2\end{array}\right),
\left(\begin{array}{c}0\\3\end{array}\right),
\left(\begin{array}{c}1\\4\end{array}\right),
\left(\begin{array}{c}3\\4\end{array}\right)
\right\}.
$$
The set $\Omega_C$ constructed by the algorithm in Lemma~\ref{alg_invariantomega} is drawn in Figure~\ref{fig_V0_neq_V0bar}. 

Straightforward computation shows that $\operatorname{dim} V_{\Omega_C}=33 > \operatorname{dim} \tilde{V}_{\Omega_C}=32$.
Thus, $\Omega_C$ will be inappropriate for further theoretical analysis. The problematic point is $(-2,1)$ which has no direct neighbour.
See Remark~\ref{rem:badOmegaC} for more details.
The Matlab code to produce Figure~\ref{fig_V0_neq_V0bar} is
\begin{verbatim}
S=getS('2_ex_V0neqV0bar_1');
Om=constructOmega(S);
plotm(Om,'k.','MarkerSize',10)
axis equal; axis([-3 9 -10 2]);
\end{verbatim}
\end{example}

\begin{figure}
	\centering
	\includegraphics[width=.99\textwidth,height=4cm,keepaspectratio]{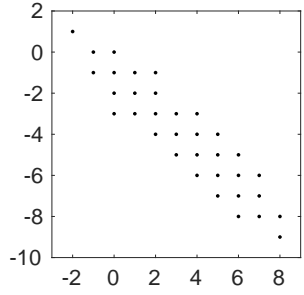}
	\caption{The set $\Omega_C$ from Example~\ref{ex_V0_neq_V0bar}.}
	\label{fig_V0_neq_V0bar}
\end{figure}

\noindent In Proposition~\ref{thm_VkequVbark_existence},
we determine a finite set $\Omega_V \subset \ZZ^s$ such that $V_{\Omega_V} = \tilde{V}_{\Omega_V}$.
The structure of $\Omega_V$ is adapted to the definition of the restricted spectral radius and makes the link 
between the two spectral radii more evident. The definition of the set $\Omega_V$ is straightforward in comparison 
to the set $\Omega_C$ from Lemma~\ref{alg_invariantomega}, but the latter is by far more efficient for 
computations in section \ref{sec:examples}. 

\noindent For simplicity of presentation and without loss of generality we make the following assumption.

\smallskip 
\noindent 
\textbf{Assumption N:} {\it 
 We assume that 
 \begin{equation}\label{equ_assumption_N}
\|M_j^{-1}\|_2<1,\ j=1, \ldots,J.
 \end{equation}}
\smallskip

\noindent 
The above assumption is true for a variety of dilation matrices considered in the literature,
but is not true e.g. for the dilation matrix $M=\left(\begin{array}{rr} 1&-2\\2&-1 \end{array} \right)$ of the
$\sqrt{3}$-subdivision. Nevertheless, Definition~\ref{def_jointly_expanding}, norm equivalences and~\cite{Rota60} guarantee the existence of $n \in \NN$
such that any product of $n$ matrices from the set $\{M_j^{-1} \;:\; j=1,\ldots,J\}$ satisfies 
Assumption~\textbf{N}. Indeed, for this matrix $\|M^{-2}\|_2<1$.  If $n>1$, we then study the
convergence of multiple subdivision defined by $\mathcal{S}^n$ instead of $\mathcal{S}$ in~\eqref{def:setS}.

\begin{proposition}\label{thm_VkequVbark_existence} For the set $\mathcal{T}$ of transition operators, there exists
 a finite set $\Omega_V \subset\ZZ^s$ such that
\begin{enumerate}
\item[$(i)$] $\ell(\Omega_V)$ is invariant under operators in $\mathcal{T}$ and
\item[$(ii)$] $V_{\Omega_V} = \tilde{V}_{\Omega_V}$.
\end{enumerate}
\end{proposition}
\begin{proof}
By Assumption~\textbf{N} in~\eqref{equ_assumption_N},
$$
 C_M=\max_{j=1,\ldots,J} \|{M_j^{-1}}\|_2 <1.
$$
Due to the finite support of the masks and finiteness of digit sets, we get finite constants
\begin{equation*}
C_a=\max \{\|\alpha\|_2\;:\; \alpha\in\operatorname{supp}(a_j),\ j=1,\ldots,J\}
\quad\text{and}\quad
C_D=\max \{ \|d\|_2\;:\; d\in D_j, \ j=1,\ldots,J\}.
\end{equation*}
We define
\begin{equation*}
 \Omega_V =\big\{x \in \RR^s\;:\; \|x\|_2 \le \frac{C_a+C_D}{1-C_M} \big\}\cap\ZZ^s.
\end{equation*}

$(i)$ Let $d\in D_j$,
$j=1,\ldots,J$. By~\eqref{equ_transition_operator}, $\mathcal{T}_{d,j} v(\alpha) \not =0$, if
$\alpha\in\ZZ^s$ is such that $M_j\alpha-\beta+d \in\operatorname{supp}a_j$ for some $\beta\in\operatorname{supp}v\subseteq \Omega_V$, or, equivalently, $\alpha \in M_j^{-1}(\operatorname{supp}a_j-d+\Omega_V)$. Thus, we obtain
\begin{equation} \label{eq:aux}
 \|{\alpha\|}_2\leq C_M \left(C_a+C_d+\frac{C_a+C_D}{1-C_M}\right)=\frac{C_a+C_D}{1-C_M} (C_M(1-C_M) + C_M)\le
\frac{C_a+C_D}{1-C_M},
\end{equation}
since $-C_M^2+2C_M \le 1$, implying $\mathcal{T}_{d,j} v \in \ell(\Omega_V)$.

$(ii)$ The dimension of $V_{\Omega_V}$ is 
$|\Omega_V|-1$, due to $V_{\Omega_V}$ being orthogonal to the vector
of all ones. To determine the dimension of $\tilde{V}_{\Omega_V}$, we consider the graph $G=(\Omega_V,W)$ with the set of edges
$$
 W=\{ (w_1, w_2) \in \Omega_V^2 \;:\; \|w_1-w_2\|_1=1\}.
$$
Using this point of view, every sequence of the form $\nabla_{\!l} \delta(\cdot-\beta)\in\tilde{V}_{\Omega_V}$, $\beta\in\ZZ^s$, $l\in\{1,\ldots,s\}$, is associated uniquely to an edge in $W$. The graph $G$ is connected, thus, there
exists a corresponding spanning tree consisting of $|\Omega_V|-1$ edges from $W$~\cite[Theorem~1.5.1]{Diestel2005}. Since any spanning tree does not contain cycles,
the set of edges of the spanning tree corresponds to a set of linearly independent sequences in $\tilde{V}_{\Omega_V}$. Thus,
$\operatorname{dim} \tilde{V}_{\Omega_V}=|\Omega_V|-1$.
\end{proof}

\begin{remark} \label{rem:badOmegaC}
 The proof of Proposition~\ref{thm_VkequVbark_existence} explains the phenomenon occurring in Example~\ref{ex_V0_neq_V0bar}.
 The graph corresponding to the set $\Omega_C$ from this example consists of two connected components. This fact
 forces $\operatorname{dim} V_{\Omega_C}> \operatorname{dim} \tilde{V}_{\Omega_C}$.
\end{remark}

\subsection{Proof of Theorem~\ref{th:main}} \label{subsec:proof_main}

\begin{proof} The proof of Theorem~\ref{th:main} generalizes the proofs of~\cite[Proposition~4.6]{Charina} from the stationary setting,
thus, we only sketch the steps of the proof.

\noindent Assumption~\textbf{S}, i.e. the sum rules, for $\mathcal{S}=\{S_j\ : \ j=1, \ldots,J\}$ guarantees the existence of the difference subdivision 
operators $S_j'$ in $\mathcal{S}'=\{S'_j\ : \ j=1, \ldots,J\}$. 
Moreover, Assumption~\textbf{S}  and~$(ii)$, by Lemma~\ref{thm_invariantone} part~$(i)$ and~\cite[Theorem~5.2]{Jia}, ensure that
$V_{\Omega}$ is a common invariant subspace of the transition matrices in $\{ T_{d,j,\Omega} \ :\ d \in D_j, \ j=1,\ldots,J\}$. 
Thus, the restrictions $T_{d,j,\Omega}|_{V_\Omega}$ of the matrices in $\{T_{d,j,\Omega} \ :\ d \in D_j, \ j=1,\ldots,J\}$ to $V_{\Omega}$ are well defined. 

\noindent Note that, by the definitions of the joint and restricted spectral radii, the claim follows from (with 
$T_\ell=T_{d_\ell,j_\ell,\Omega}$ for $d_\ell \in D_{\!j_\ell}$,
$j_\ell \in \{1, \ldots, J\}$ and $\ell=1, \ldots, n$, $n \in \NN$)
$$
  C_1 \ \max_{T_\ell \in \mathcal{T}} \  
  \| T_n \cdots T_1|_{V_\Omega}\|_\infty 
  \leq 
  \|S'_{j_n} \cdots S'_{j_1} |_\nabla\|_\infty 
	\leq 
	C_2 \ \max_{T_\ell \in \mathcal{T}} \ 
	\|T_n \cdots T_1|_{V_\Omega}\|_\infty
$$
with some constants $C_1, C_2 >0$. To determine $C_1$ and $C_2$, we first use the assumption $(ii)$ and an argument similar to the one of~\cite[Lemma~4.5]{Charina} which implies that 
$$
 C_3 \ \max_{T_\ell \in \mathcal{T}} \ 
 \| T_n \cdots T_1|_{V_\Omega}\|_\infty  
 \leq
 \| S'_{j_n} \cdots S'_{j_1} \nabla \delta  \|_\infty 
 \leq  
 C_4 \ \max_{T_\ell \in \mathcal{T}} \ 
 \| T_n \cdots T_1|_{V_\Omega}\|_\infty
$$
with $C_3=|\Omega|^{-2}$ and $C_4=1$. Then, for $K$ in~\eqref{def:K}, define $\Omega_K=([-1,1]^s-K) \cap \ZZ^s$. Due to $\delta \in \ell_\infty(\Omega_K)$ and
$\|\nabla \delta\|_\infty=1$, by  Proposition~\ref{thm_different_RN}, we get
$$
  \|S'_{j_n} \cdots S'_{j_1} |_\nabla\|_\infty  \ge \| S'_{j_n} \cdots S'_{j_1} \nabla \delta  \|_\infty.
$$
Thus, $C_1=C_3$. Moreover, using 

\begin{equation*}
{\nabla} c=
{\nabla}\left(\sum_{\alpha\in\ZZ^s} c(\alpha)\delta(\cdot-\alpha) \right)=
\sum_{\alpha\in\ZZ^s} c(\alpha) {\nabla}\delta(\cdot-\alpha),
\end{equation*}
for the maximizing sequence $c \in \ell_\infty(\Omega_K)$
from Proposition~\ref{thm_different_RN}, we obtain
$$
   \|S'_{j_n} \cdots S'_{j_1} |_\nabla\|_\infty  \le C_5 \| S'_{j_n} \cdots S'_{j_1} \nabla \delta  \|_\infty
$$
with $C_5=|\Omega_K| \cdot \|c\|_\infty$. Thus, $C_2= C_5$.
\end{proof}

\section{Examples} \label{sec:examples}

\noindent Example~\ref{ex_mult1} shows that already in the univariate, stationary case the assumption $V_{\Omega} =\tilde{V}_{\Omega}$ in Theorem~\ref{th:main} is crucial.

\begin{example} \label{ex_mult1} We consider the stationary subdivision scheme with dilation factor
$M=2$ and mask $a \in \ell_0(\ZZ)$ whose non-zero elements are given by
$$
 a(0)=\frac{1}{2}, \quad a(3)=1 \quad \text{and}\quad a(6)=\frac{1}{2}.
$$
It is well known that this subdivision scheme does not converge, although there is a continuous, piecewise linear,
compactly supported on $[0,6]$, solution $\phi$ of the corresponding refinement equation.  For illustration purposes, we choose the digit set $D=\{0,3\}$. The algorithm in Lemma~\ref{alg_invariantomega} generates
the set $\Omega_C=\{0,3\}$ and, by Definition~\ref{def_V0}, we have 
$\operatorname{dim}V_{\Omega_C} =2 > \operatorname{dim}\tilde{V}_{\Omega_C}=0$. The set $\Omega_V$, with the property $\operatorname{dim}V_{\Omega_V} = \operatorname{dim}\tilde{V}_{\Omega_V}$, can be chosen, in this case, to be $\Omega_V=\{-2,\ldots,5\}$. We make this
choice for simplicity reasons, the set $\Omega_V$ from Proposition~\ref{thm_VkequVbark_existence} would be
of size $61$. The corresponding transition matrices
$T_{d,\Omega_V}$, $d \in D$, have the following block form
$$
 T_{0,\Omega_V}=\frac{1}{2}
 \left(\! \begin{array}{cccc} T_{0,\Omega_C}&0 \\0& T_{0,\Omega'}\end{array}\!\right) 
 \quad\text{and}\quad
 T_{3,\Omega_V}=\frac{1}{2}
 \left(\! \begin{array}{cccc} T_{3,\Omega_C}&0 \\0& T_{3,\Omega'}\end{array}\!\right),
 \quad
 \Omega'=\Omega_V \setminus \Omega_C =
 \{-2,-1,1,2,4,5\},
$$
with
$$
 T_{0,\Omega_C}=\frac{1}{2}\left( \begin{array}{cc}1&0\\1&2\end{array}\right), \quad
 T_{3,\Omega_C}=\frac{1}{2}\left( \begin{array}{cc}2&1\\0&1\end{array}\right),
$$
and
$$
 T_{0,\Omega'}=\frac{1}{2}\left( \begin{array}{cccccc}0&0&0&0&0&0\\1&0&0&0&0&0\\0&2&0&1&0&0
 \\1&0&2&0&1&0\\0&0&0&1&0&2\\0&0&0&0&1&0\end{array}\right), \quad
 T_{3,\Omega'}=\frac{1}{2}\left( \begin{array}{cccccc}0&1&0&0&0&0\\2&0&1&0&0&0\\1&0&2&0&1&0
 \\0&0&1&0&2&0\\0&0&0&0&0&1\\0&0&0&0&0&0\end{array}\right).
$$
By Lemma~\ref{thm_invariantone} part~$(ii)$, the space $\ell(\Omega_V)$ is $\mathcal{T}$ invariant. Thus, by
Theorems~\ref{thm_RSR_convergence} and \ref{th:main}, due to $\operatorname{dim}V_{\Omega_V} = \operatorname{dim}\tilde{V}_{\Omega_V}$ and 
$\rho(\{T_{d,\Omega_V}|_{V_{\Omega_V}} : d \in D\})=1$, we get the correct answer that the scheme is not convergent. On the contrary, 
$\rho(\{T_{d,\Omega_C}|_{V_{\Omega_C}} : d \in D\})=1/2$ is misleading. Here, we used the invariant polytope algorithm from~\cite{GP1} for our computations.
\end{example}

\noindent Multivariate example~\ref{ex_mult2} illustrates the properties of multiple subdivision $\mathcal{S}^\NN$
in the case $V_{\Omega_C}=\tilde{V}_{\Omega_C}$.

\begin{example} \label{ex_mult2}
We consider the set $\mathcal{S}=\{(a_j,M_j)\ : \ j=1,2\}$ of subdivision operators from Example~\ref{ex_supp_of_blf} with the
corresponding digit sets
\begin{equation*}
D_1=\left\{
\left(\begin{array}{r}0\\0\end{array}\right),
\left(\begin{array}{r}1\\\!-1\end{array}\right),
\left(\begin{array}{r}1\\0\end{array}\right)
\right\}
\quad\text{and}\quad
D_2=\left\{
\left(\begin{array}{r}0\\0\end{array}\right),
\left(\begin{array}{r}0\\\!-1\end{array}\right),
\left(\begin{array}{r}1\\0\end{array}\right)
\right\}.
\end{equation*}

\begin{figure} \label{fig:Omega_eample_section}
\centering
 \includegraphics[width=.99\textwidth,height=4cm,keepaspectratio]{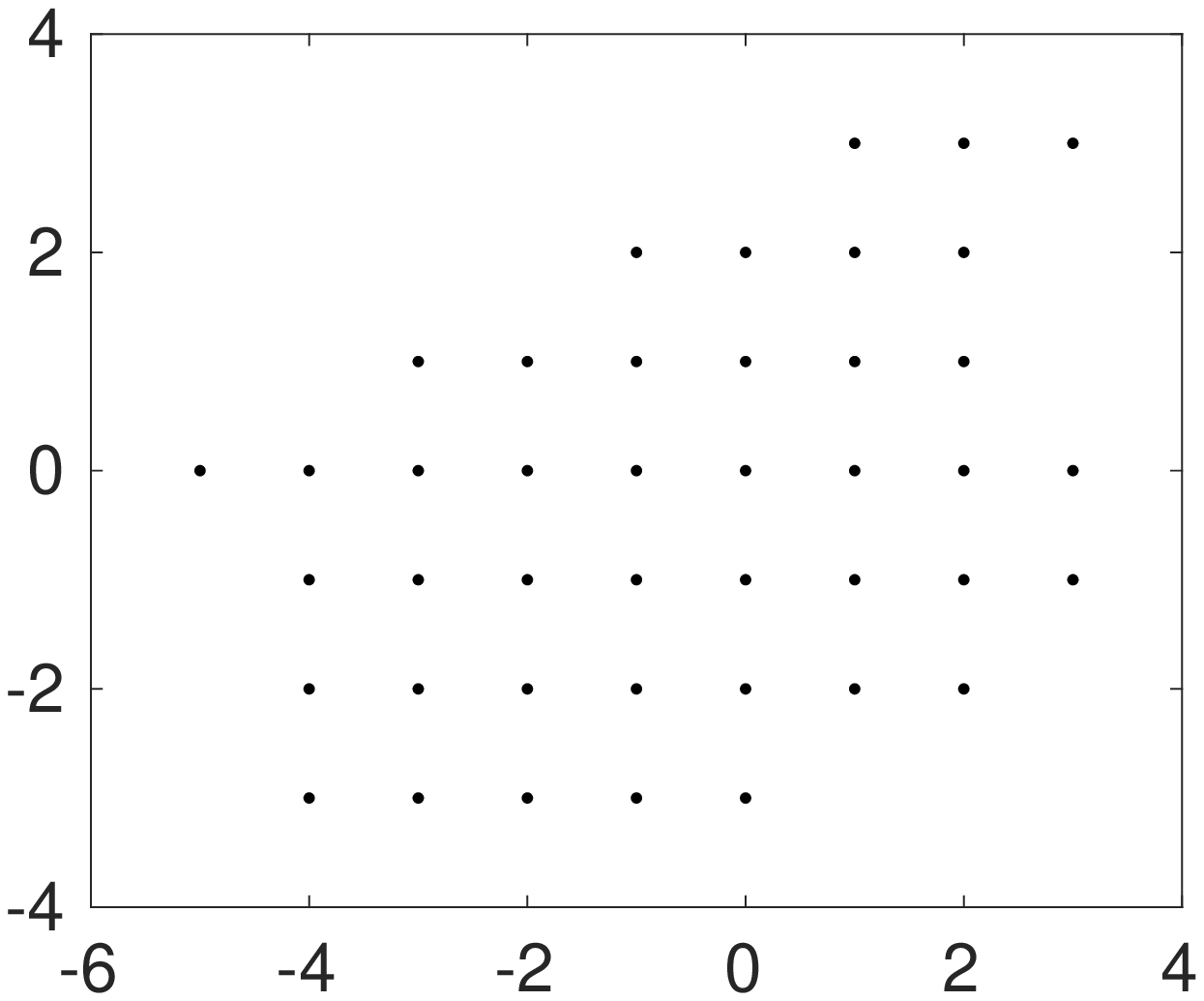}
 \caption{Set $\Omega_C$ from Example~\ref{ex_mult2}.}
\label{fig_mult}
\end{figure}
\noindent Note that the masks $a_1=a_2$ satisfy Assumption~\textbf{S}, i.e. the sum rules.
The set $\Omega_C$ computed by the algorithm in Lemma~\ref{alg_invariantomega} is given on Figure
\ref{fig:Omega_eample_section}.
Note that $\operatorname{dim}V_{\Omega_C}=\operatorname{dim}\tilde{V}_{\Omega_C}$.
By the invariant polytope algorithm from~\cite{GP1} we obtain
\begin{equation*}
 \rho(\{T_{d,\Omega_C,j}|_{V_{\Omega_C}},\ d\in D_{j},\ j=1,2\})= 
\rho\left(
T_{{\tiny\left(\!\!\!\!\begin{array}{c}0\\0\end{array}\!\!\!\!\right)}, \Omega_C,1} 
T_{{\tiny\left(\!\!\!\!\begin{array}{c}1\\0\end{array}\!\!\!\!\right)}, \Omega_C,2} 
\right)^{1/2} = 
0.8971\!\ldots
\end{equation*}
Therefore, by Theorems~\ref{thm_RSR_convergence} and~\ref{th:main}, $\mathcal{S}^\NN$ is convergent.
By~\cite[Theorem~1 and Remark~3]{CP2017} and~\cite[Proposition~3.27]{CHM}, the critical H\"older exponent $\alpha$ of the stationary
subdivision scheme $(S_1)_{\ell \in \NN}$ with the anisotropic dilation matrix $M_1$ satisfies $\alpha \in [0.3446\!\ldots,1]$.
For the stationary subdivision scheme $(S_2)_{\ell \in \NN}$ with the isotropic dilation matrix $M_2$, 
we obtain $\alpha=1$.  For the stationary subdivision scheme $(S_1 S_2)_{\ell \in \NN}$ with the isotropic dilation matrix
$M_1M_2$, we get $\alpha=0.1977\!\ldots\ $.
\end{example}

\section{Acknowledgement}
 Both authors are sponsored by the Austrian Science Foundation (FWF) grant P28287-N35.

\section{Appendix}

\noindent For completeness, we provide the MATLAB code of the algorithm described in
Lemma~\ref{alg_invariantomega} in copy-paste-able format.
The code, together with everything needed to execute the code snippets in this paper, is available for download from 
\href{http://tommsch.com}{http://tommsch.com}
.

\begin{verbatim}
function [ Om ] = Omega(a, M, D, Om)
% a, M, D: cell vector of masks, dilation matrices and digit sets (as column vectors)
% Om: (Optional) the starting set
% ex: Omega({[1:3 2 1]/3,[1:3 2 1]/3},{[2 -1;1 -2],[1 1;1 -2]},{[0:2;0 0 0],[0:2;0 0 0]})
% Out: -[4 4 4 3 3 3 3 3 2 2 2 2 2 1 1 1 1 0 0;4 3 2 4 3 2 1 0 3 2 1 0 -1 2 1 0 -1 0 -1];
dim=size(M{1},1);                         %the dimension
if(nargin==3); Om=zeros(dim,1); end       %if Omega is not given, set it to zero
while(true)
   sizebefore=size(Om,2);                 %used to check if elements where added to Omega
   for j=1:size(a,1)                      %iterate through all subdivision operators
      OmN=M{j}\setplus(supp(a{j},dim),Om,-D{j}); %compute new possible entries
      OmN=round(OmN(:,sum(abs(OmN-round(OmN)),1)<.5/abs(det(M{j})))); %round to integers
      Om=unique([Om OmN]','rows')';       %remove duplicates
   end
   if(size(Om,2)==sizebefore); break; end %if no elements were added, terminate
end

function [ X ] = setplus( varargin )
% setplus(A,B) = { x=a+b : a in A, b in B}, operates column wise
% ex: setplus([1 2; 1 0],[0 -1; -1 -1]); %Output: [0 1 1 2;0 -1 0 -1]
sze=size(varargin,2);     %number of sets
X=varargin{sze};          %the output set
for i=sze-1:-1:1          %iterate through all sets
   A=varargin{i};         %the set to be added
   X=repmat(A,1,size(X,2))+reshape(repmat(X,size(A,2),1),size(A,1),[]); %add the set
   X=unique(X','rows')';  %remove duplicates
end

function [ L ] = supp(a, dim)
% returns the support of an array. First entry is supposed to have index (0,0,...,0)
% ex: supp([1 1; 0 1],2) %Output: [0 0 1; 0 1 1];
L=zeros(dim,nnz(a));      %output variable
CO=cell(1,dim);           %dummy-variable to do calculation with indices
j=1;                      %index-variable for the columns of D
for i=1:numel(a)          %iterate through all elements of the masks
   if(a(i)~=0)            %if the element is nonzero, save the indices
      [CO{:}]=ind2sub(size(a),i); %get the indices
      L(:,j)=[CO{:}]'-1;  %change to zero-based indexing, add converted cell to vector
      j=j+1;              %increase counter
   end
end
\end{verbatim}

\newpage
\section{References}

\end{document}